\documentclass[a4paper,10pt]{amsart}
 
\usepackage[english]{babel}
\usepackage{dsfont} 
\usepackage{graphicx}
\usepackage{color}
\usepackage{enumerate}
\usepackage{array,multirow}

\allowdisplaybreaks

\usepackage{amssymb} 
\usepackage{amsthm} 
\usepackage{dsfont}

\newtheorem{thm}{Theorem}
\newtheorem{dfn}[thm]{Definition}
\newtheorem{lem}[thm]{Lemma}

\newtheorem{exm}[thm]{Example}

\newtheorem{prop}[thm]{Proposition}
\newtheorem{rem}[thm]{Remark}

\newcommand{\cR}{\mathcal{R}}
\newcommand{\dR}{\mathds{R}}
\newcommand{\dC}{\mathds{C}}
\newcommand{\ii}{\mathrm{i}}

\newcommand{\cC}{\mathcal{C}\,}

\newcommand{\gA}{\mathfrak{A}\,}

\newcommand{\cD}{\mathcal{D}\,}
\newcommand{\cH}{\mathcal{H}\,}

\newcommand{\ov}{\overline}

\newcommand{\R}{\mathrm{Re}}
\newcommand{\I}{\mathrm{Im}}

\newcommand{\cK}{\mathcal{K}}
\newcommand{\cN}{\mathcal{N}}

\newcommand{\cU}{\mathcal{U}}

\newcommand{\gT}{\mathfrak{T}}
\newcommand{\gC}{\mathfrak{C}}

\author{Konrad Schm\"udgen}
\address{University of Leipzig, Mathematical Institute, Augustusplatz 10/11, D-04109 Leipzig, Germany}
\email{\tt schmuedgen@math.uni-leipzig.de}

\date{}
\begin{document}

\begin{title}{Adjoint Pairs and Unbounded Normal Operators}
\end{title}
\date{\today}

\begin{abstract}
An adjoint pair is a pair of densely defined linear operators $A, B$ on a Hilbert space such that
$\langle Ax,y\rangle=\langle x,By\rangle$ for $x\in \cD(A), y \in \cD(B).$ We consider adjoint pairs for which $0$ is a regular point for both operators and associate a boundary triplet to such an adjoint pair. Proper extensions of the operator $B$ are in one-to-one correspondence $T_\cC\leftrightarrow \cC $ to closed subspaces $\cC$ of $\cN(A^*)\oplus\cN(B^*)$. In the case when   $B$ is formally normal and $\cD(A)=\cD(B)$, the normal operators $T_\cC$ are characterized. Next we assume that $B$ has an extension to a normal operator with bounded inverse. Then the normal operators $T_\cC$ are described and the case when $\cN(A^*)$ has dimension one is treated.   
\end{abstract}
\maketitle


\textbf{Key  words:} adjoint pair, boundary triplet, formally normal operator, unbounded normal operator

\section{Introduction}
This paper deals with various notions and constructions in unbounded operator theory on Hilbert space.
The basic objects studied in this paper are {\it adjoint pairs}  $\{A,B\}$ (see Definition \ref{defadjoint}) of densely defined operators $A,B$ on a Hilbert space $\cH$ for which the number $0$ is a regular point of $A$ and $B$.

In Section \ref{simple}, we  use a  result of M.I. Vishik \cite{Vi} (stated as Theorem \ref{btab} below) to associate to such a pair a boundary triplet (see Definition \ref{btri}) for the  operator matrix 
  \begin{gather}
\gA=\left(
\begin{array}{ll}
0 & A \\
B & 0
\end{array}\right)
\end{gather}
acting as  symmetric operator with domain $\cD(B)\oplus\cD(A)$ on $\cH\oplus \cH$ (Theorem \ref{boundtripad}). Then the theory of boundary triplets allows one to decribe the proper extensions of the symmetric operator $\gA$ in terms of closed relations.

In the remaining Sections  \ref{formn1}--\ref{formn3} we assume in addition that the operator $B$ is {\it formally normal} and $\cD(A)=\cD(B)$. Our aim is to study normal extensions of the operator $B$.

The proper extensions of $B$ (that is, closed operators $T$ satisfying $B\subseteq T\subseteq A^*)$ can be described in terms of closed subspaces $\cC$ of the Hilbert  space $\cN(A^*)\oplus\cN(B^*)$. Let $T_\cC$ denote the corresponding operator. 

In Section \ref{formn1},  the normality of the operator $T_\cC$ is characterized in terms of the subspace $\cC$
 (Theorems \ref{genform} and \ref{graphform}). In Section \ref{formn2} and \ref{formn3} we assume that the formally normal operator $B$ admits an extension to a normal operator $R^*$ with bounded inverse. Then there exists a unitary operator $W$ satisfying $R^{-1}W= (R^*)^{-1}$ which leads to simplifications of the normalcy criteria for the operator $T_\cC$. The case when $\cN(A^*)$ has dimension one is  treated in Section \ref{formn3} and the result is  
 Theorem \ref{onedimcase}.
 
  {\bf Throughout the whole paper, $\{A,B\}$ denotes an  adjoint pair such that $0$ is a regular point for $A$ and $B$.}
 
 Let us add a few bibliographical comments and hints. Adjoint pairs are treated in \cite{Vi} and \cite[Section III.3]{EE}. Boundary triplets have been invented by Kochubei \cite{Ko} and Bruk \cite{Bk}; a fundamental paper on boundary triplets is \cite{DM}. Boundary triplets associated with adjoint pairs were constructed and studied in \cite{MM}. Pioneering work on formally normal operators and their extensions to normal operators was done by Biriuk and Coddington \cite{BC}, \cite{Cd2}. The existence of formally normal operators which have no normal extension was discovered by Coddington \cite{Cd1}; a very simple example can be found in \cite{sch86}. Unbounded normal operators have been extensively studied by Stochel and Szafraniec, see e.g. \cite{SS1}, \cite{SS2}, \cite{SS3}. Concerning the theory of unbounded operators on Hilbert space we refer to the author's graduate text \cite{sch12}, see also \cite{EE}.
 \section{Some operator-theoretic notions}
 In this short section we collect a few  concepts and notations from operator theory which are crucial in what follows.
 
 Let $T$ be a linear operator on a Hilbert space. We denote by $\cD(T)$ its domain, by $\cR(T)$ its range und by $\cN(T)$ its kernel. 
 
 The symbol $\dot{+}$ refers to the direct sum of vector spcaes.
 
 The  bounded operator on $\cH$ are denoted by ${\bf B}(\cH)$.
 
 A number $\lambda\in \dC$ is called \emph{regular}  for  $T$ if there exists a constant $\gamma>0$ (depending on $\gamma$ in general) such that\begin{align}
\|(T-\lambda\, I)\varphi\|\geq \gamma\|\varphi\|\quad \textit{for}\quad \varphi\in \cD(T).
\end{align}

A densely defined operator $A$ is called {\it formally normal} if $\cD(A)\subseteq \cD(A^*)$ and 
\begin{align}\label{nory}
\|Ax\|=\|A^*x\|\quad {\rm for}~~~x\in \cD(A).
\end{align} By polarization, condition (\ref{nory}) implies that 
\begin{align}\label{polar}
\langle Ax, Ax'\rangle= \langle A^*x,A^*x'\rangle~~~{\rm for}~~~ x,x'\in \cD(A)=\cD(A^*).
\end{align}
A formally normal operator $A$ is called {\it normal} if $\cD(A)=\cD(A^*)$.

Next we recall the notion of a {\it boundary triplet}.

\begin{dfn}\label{btri}  Suppose that $T$ is a  densely defined symmetric operator on $\cH$. A {\em boundary triplet} for\, $T^*$\, is a triplet\, $(\cK, \Gamma_0, \Gamma_1)$\, of a
Hilbert space $(\cK, (\cdot,\cdot))$ and linear mappings $\Gamma_0:\cD(T^\ast) \rightarrow \cK$ and
$\Gamma_1:\cD(T^\ast) \rightarrow \cK$ such
that:\begin{itemize}
\item[(i)]\label{con1}~ $\langle T^\ast x,y \rangle - \langle x,T^\ast y \rangle = (\Gamma_1 x, \Gamma_0 y) - 
(\Gamma_0 x, \Gamma_1 y)$ for $x,y \in \cD(T^\ast )$,
\item[ (ii)]\label{con2}\, ~the mapping $\cD(T^\ast )\ni x \mapsto (\Gamma_0 x, \Gamma_1 x) \in \cK \oplus \cK$ is surjective. 
\end{itemize}
\end{dfn}

\section{Adjoint pairs and boundary triplets}\label{simple} 

The main concept occurring in this paper is the following.

\begin{dfn}\label{defadjoint}
An \emph{adjoint pair} is a pair $\{A,B\}$ of densely defined linear operators $A$ and $B$ on  Hilbert space $\cH$ such that
\begin{align}\label{ab1}
\langle Ax,y\rangle=\langle x,By\rangle\quad \textit{for}\quad x\in \cD(A),\, y \in \cD(B).
\end{align}
\end{dfn}
Clearly, (\ref{ab1}) is equivalent to the relations
\begin{align}\label{ab2}
A\subseteq B^*\quad {\rm and}\quad B\subseteq A^*
\end{align}
Thus, any pair  of densely defined 
operators $A,B$ satisfying (\ref{ab2}) is an adjoint pair. 

In the literature, ``adjoint pairs" often appear as ``dual pairs"  (for instance, in \cite{MM}). Since the latter  notion is used in different context in other parts of mathematics, we prefer to speak about ``adjoint pairs" (as in \cite{EE}).

We mention two examples.
\begin{exm}
Suppose $A$ is a densely defined closed operator. Then $\{A,A^*\}$ and $\{A^*,A\}$ are adjoint pairs. Note that in both cases we  have equalities in (\ref{ab2}). 
\end{exm}

\begin{exm} Suppose $T$ is  a densely defined closed symmetric operator and $\alpha\in \dC$. Then the operators $A_0:=T+\alpha I$  and $B_0:=T^*+\ov{\alpha} I$ form an adjoint pair.

 Likewise, $A:=T+\alpha I$  and $B:=(T^*+\ov{\alpha} I)\lceil \cD(T)=T+\ov{\alpha} I$ are an adjoint pair. 
Set $a=\R\, \alpha$ and $b=\I\, \alpha$. For $x\in \cD(T)$ we compute
\begin{align}\label{abregex}
\|Ax\|^2=\|(T+a I)x\|^2+ b^2\|x\|^2=\|Bx\|^2.
\end{align}
Hence  $A$ and $B$ are formally normal operators. Further, if $b\neq 0$, it follows from (\ref{abregex}) that $0$ is a regular point for $A$ and $B$. 
\end{exm}

The considerations of this paper are essentially based on an  important theorem of Vishik \cite[Theorems 1 and 2]{Vi}. It is an extension of a result of Calkin \cite{Ck} for symmetric operators. We state this result as Theorem \ref{btab} and add an number of useful formulas.  

A crucial part is the existence of the operator $R$ with bounded inverse; a nice proof of this assertion can be also found in \cite[Theorem 3.3]{EE}. 
\begin{thm}\label{btab}
Suppose $\{A,B\}$ is a adjoint pair and $0$ is a regular number for  the operators $A$ and $B$, that is, exists a constant $\gamma>0$ such that
\begin{align}\label{reggamma}
\|Ax\|\geq \gamma\|x\| ~~~\textit{and}~~~ \|By\|\geq \gamma\|y\|\quad \textit{for}\quad x\in \cD(A), y \in \cD(B).
\end{align}
Then there exists a closed operator $R$ on $\cH$ such that $R$ and $R^*$ have  inverses  $R^{-1}\in {\bf B}(\cH) $, $(R^*)^{-1}\in {\bf B}(\cH)$, 
\begin{align}\label{das0}
 A\subseteq R\subseteq B^*,\quad B\subseteq R^*\subseteq A^*
 \end{align} and
\begin{align}\label{das1}
\cD(A^*)&=\cD(B)\dot{+} (R^*)^{-1}\cN(B^*) \dot{+}\cN(A^*),\\ \cD(B^*)&=\cD(A)\dot{+} R^{-1}\cN(A^*) \dot{+}\cN(B^*).\label{das2}
\end{align} 
For~ $ x_0\in \cD(A),\, y_0\in \cD(B),\, u\in \cN(A^*),\, v\in\cN(B^*)$, we have
\begin{align}\label{das3}
A^*(x_0+(R^*)^{-1}v+u)= Bx_0+v, \quad B^*(y_
0+R^{-1}u+v)=Ay_0+u.
\end{align}
\end{thm}

Let us adopt the following notational convention: Elements of $\cN(A^*)$ are denoted by $u, u',u_1,u_2,u_1',u_2'$, while symbols $v,v',v_1,v_2,v_1',v_2'$ always refer to vectors of $\cN(B^*)$.

As throughout,  $\{A,B\}$ denotes an adjoint pair such that $0$ is a regular point for the operators $A$ and $B$.

We define an operator $\gA$ with domain $\cD(\gA)=\cD(B)\oplus\cD(A)$ on the direct sum Hilbert space $\cH\oplus \cH$ by the operator matrix
  \begin{gather}\label{eq_matr_abbC} 
\gA=\left(
\begin{array}{ll}
0 & A \\
B & 0
\end{array}\right).
\end{gather}
From (\ref{ab1}) it follows at once that the operator $\gA$ is symmetric. It is easily verified that the adjoint operator $\gA^*$ has the domain $\cD(\gA^*)=\cD(A^*)\oplus\cD(B^*)$ and is given by the  matrix \begin{gather} 
\gA^*=\left(
\begin{array}{ll}
0 & B^* \\
A^* & 0
\end{array}\right).
\end{gather}
Let $(x,y), (x',y')\in \cD(\gA^*)$. Then $x,x'\in \cD(A^*)$ and $ y,y'\in \cD(B^*)$. Therefore,  by (\ref{das1}) and (\ref{das2}),  $x,y,x',y'$ are of the form 
\begin{align}\label{xform}
&x=x_0+(R^*)^{-1}v_1+u_1,~ x'=x_0'+ (R^*)^{-1}v_1' +u_1',\\& y=y_0+R^{-1}u_2+v_2,~ y'=y_0'+R^{-1}u_2'+v_2',\label{yform}
\end{align}
with $x_0,x_0'\in \cD(B),\, y_0,y_0'\in \cD(A),\, v_1,v_2,v_1',v_2'\in \cN(B^*),\, u_1,u_2, u_1', u_2'\in\cN(A^*)$.
Using equation (\ref{das0}) we derive
\begin{align}\label{das4}
&\langle Ay_0,(R^*)^{-1}v_1'\rangle =\langle Ry_0,(R^*)^{-1}v_1'\rangle=\langle y_0,R^*(R^*)^{-1}v_1'\rangle= \langle y_0,v_1'\rangle\end{align}
and similarly
\begin{align}\label{das5} \langle (R^*)^{-1} v_1,Ay_0'\rangle=\langle v_1,y_0'\rangle.
\end{align}
Replacing (\ref{das0}) by (\ref{das1}) the same reasoning  yields
\begin{align}\label{das4a}
  \langle Bx_0,R^{-1}u_2'\rangle=\langle x_0,u_2'\rangle,~~~ \langle R^{-1}u_2,Bx_0'\rangle=\langle u_2,x_0'\rangle.
\end{align}
Further, since $ u_1', u_1\in\cN(A^*)$ and $v_2',v_2'\in \cN(B^*),$ we have
\begin{align}\label{das6}
\langle Ay_0,u_1'\rangle=\langle u_1,Ay_0'\rangle=\langle Bx_0,v_2'\rangle=\langle v_2,Bx_0'\rangle=0.
\end{align}
Now we apply the preceding formulas (\ref{ab1}), (\ref{das3}), (\ref{das4}), (\ref{das5}), (\ref{das4a}), (\ref{das6}) and compute
\begin{align}
&\langle \gA^*(x,y),(x',y')\rangle- \langle (x,y),\gA^*(x,y')\rangle \nonumber
\\&= \langle B^*y,x'\rangle +\langle A^*x,y'\rangle-\langle x, B^*y'\rangle-\langle y,A^*x'\rangle\nonumber
\\ &=\langle Ay_0+u_2,x_0'+ (R^*)^{-1}v_1' +u_1'\rangle +\langle Bx_0+v_1,y_0'+R^{-1}u_2'+v_2'\rangle\nonumber
 \\& ~~~ -\langle x_0+(R^*)^{-1}v_1+u_1,A y_0'+u_2'\rangle- \langle y_0+R^{-1}u_2+v_2,Bx_0'+v_1'\rangle\nonumber
\\ &= \langle Ay_0,x_0' \rangle +\langle y_0,v_1'\rangle +\langle u_2,x_0'\rangle+ \langle u_2,(R^*)^{-1}v_1'\rangle+\langle u_2,u_1'\rangle\nonumber
 \\ & ~~~ + \langle Bx_0,y_0'\rangle +\langle x_0,u_2'\rangle +\langle v_1,y_0'\rangle+\langle v_1,R^{-1}u_2'\rangle +\langle v_1,v_2'\rangle\nonumber
  \\ & ~~~ -[ \langle x_0,Ay_0'\rangle +\langle v_1,y_0' \rangle+ \langle x_0,u_2'\rangle+ \langle (R^*)^{-1}v_1,u_2'\rangle +\langle u_1, u_2'\rangle]\nonumber
  \\ & ~~~ -[ \langle y_0,Bx_0'\rangle +\langle u_2,x_0'\rangle + \langle y_0,v_1'\rangle +\langle R^{-1}u_2,v_1'\rangle +\langle v_2,v_1' \rangle]\nonumber \\ &
  =\langle u_2,u_1'\rangle+\langle v_1,v_2'\rangle-\langle u_1, u_2'\rangle -\langle v_2,v_1' \rangle.\label{das7}
\end{align}

Next we introduce an auxiliary Hilbert space $$\cK=\cN(A^*)\oplus\cN(B^*),$$ with scalar product $(\cdot,\cdot)$ defined by
\begin{align}((u,v),(u',v'))=\langle u,u'\rangle+\langle v,v'\rangle,\quad u,u'\in \cN(A^*), ~ v,v'\in \cN(B^*),
\end{align}
and linear mappings $\Gamma_0:\cD(\gA^*)\mapsto \cK$ and $\Gamma_1:\cD(\gA^*)\mapsto\cK$  by
\begin{align}\label{gammadef}
\Gamma_0(x,y)=(u_1,v_1) \quad {\rm and}\quad \Gamma_1(x,y)=(u_2,- v_2),
\end{align}
where $x,y$ are of the form (\ref{xform}) and (\ref{yform}). Then
\begin{align} 
( &\Gamma_1 (x,y),\Gamma_0(x',y'))-( \Gamma_0 (x,y),\Gamma_1 (x',y'))\nonumber\\&=((u_2,-v_2),(u_1',v_1'))-((u_1,v_1),(u_2',-v_2'))\nonumber\\ &=\langle u_2,u_1'\rangle-\langle v_2,v_1'\rangle - \langle u_1,u_2'\rangle+
\langle v_1,v_2'\rangle \label{das8}
\end{align}
for $(x,y), (x',y')\in \cD(\gA^*)$. Comparing (\ref{das8}) with (\ref{das7}) we get
\begin{align*}
\langle \gA^*(x,y),(x',y')\rangle- \langle (x,y),\gA^*(x,y')\rangle= ( \Gamma_1 (x,y),\Gamma_0(x',y'))-( \Gamma_0 (x,y),\Gamma_1 (x',y')),
\end{align*}
which is condition (i) of Definition \ref{btri} for $T=\gA$. Condition (ii) of  Definition \ref{btri} is obvious from the description of domains $\cD(A^*)$ and $\cD(B^*)$ given in Proposition \ref{btab}. Summarizing the preceding  we have proved the following
\begin{thm}\label{boundtripad}
Suppose that $\{A,B\}$ is an adjoint pair such that $0$ is a regular point for $A$ and $B$. Then the triplet $(\cK,\Gamma_0,\Gamma_1)$ of the Hilbert space $\cK=\cN(A^*)\oplus\cN(B^*)$ and the mappings $\Gamma_0$ and $\Gamma_1$, defined in equation (\ref{gammadef}), is a boundary triplet for the operator $\gA^*$.
\end{thm}
\begin{rem}
Suppose $A$ and $B$ are operators on a Hilbert space and $\alpha\in \dC$. Set 
\begin{align*}
A':=A-\alpha I ~~ {\rm and}~~ B':=B-\ov{\alpha}\, I.
\end{align*} From Definition \ref{defadjoint} it follows at once that $\{A,B\}$ is an adjoint pair if and only if $\{A', B'\}$ is an adjoint pair.  Obviously, $0$ is a regular number for $A$ and $B$ if and only if $\alpha$ is a regular number for $A'$ and\, $\ov{\alpha}$ is a regular number for $B'$. Further, $A$ and $B$ are formally normal (resp. normal) if and only if $A'$ and $B'$ are formally normal (resp. normal). Using these facts we can treat adjoint pairs $\{A',B'\}$ for which $\alpha$ is a regular number for $A'$ and $\ov{\alpha}$ is a regular number for $B'$ by reducing them to pairs $\{A,B\}$ studied in this paper. Note that in the corresponding formulas we have to replace $\cN(A)$ by $\cN(A'+\alpha\, I)$ and $\cN(B)$ by $\cN(B'+\ov{\alpha}\, I).$
\end{rem}

Next we restate some  facts from the theory of boundary triplets adapted to the present situation (see e.g. \cite[Section 14.2]{sch12}). Recall that a closed operator $\gT$  on $\cH_2:=\cH\oplus \cH$ is called a {\it proper extension} of the symmetric operator $\gA$ if $$\gA\subseteq \gT\subseteq \gA^*.$$
A {\it closed relation} on $\cK_2$ is a closed linear subspace of $\cK_2:=\cK\oplus\cK$. 
\begin{lem}\label{boungc}
Suppose $\gC$ is a closed relation on $\cK_2=\cK\oplus \cK$. Then there exists a unique proper extension $\gA_\gC$ of $\gA$ defined by $\gA_\gC:=\gA^*\lceil \cD(\gA_\gC)$, where
\begin{align}
\cD(\gA_\gC)=\{ \cD(\gA^*): (\Gamma_0(x,y),\Gamma_1(x,y))\in \gC\}.
\end{align}
Each proper extension of $\gA$ is of this form. Further, the extension $\gA_\gC$ of $\gA$ is self-adjoint if and only the relation $\gC$ is self-adjoint.
\end{lem}
\begin{proof}
\cite[Proposition 14.17]{sch12}.\end{proof}Clearly, $\gT$ is a proper extension of $\gA$ if and only if there are closed operators $S,T$ on $\cH$ such that $$A\subseteq S\subseteq B^*, B\subseteq T\subseteq A^*,$$ and
\begin{gather}\label{matrixt} 
\gT=\left(
\begin{array}{ll}
0 & S \\
T & 0
\end{array}\right).
\end{gather}
It is clear that $\gT$ is self-adjoint operator on $\cH_2$ if and only if $T=S^*$. 

For an adjoint pair $\{A,B\}$,   a   closed operator $T$ on $\cH$ satisfying $B\subseteq T\subseteq A^*$ is called a {\it proper extension} of $B$. Likewise, by a proper extension of $A$ we mean a closed operator $S$ such that $A\subseteq S\subseteq B^*$. 

Then, as discussed in the paragraph before last,  self-adjoint extensions of $\gA$ on $\cH\oplus\cH$ are in one-to-one correspondence to proper extensions $T$ of the operator $B$ on $\cH$, and equivalently, to proper extensions $S$ of $A$ on $\cH$. These operators $T$ and $S$ will be studied in the next section. 

The passage to $2\times2$ operator matrices is an old and powerful  trick in operator theory which was  used in many papers and  different contexts, see e.g. \cite{A} or \cite{GS}.

\section{Formally normal operators and normal operators}\label{formn1}
In this section we continue the considerations of the previous section and assume in addition that  $B$ is a {\bf closed formally normal operator and $\cD(A)=\cD(B)$. }

Recall that $B$ is  formally normal means that $\cD(B)\subseteq \cD(B^*)$ and  $\|Bx\|=\|B^*x\|$ for $x\in \cD(B)$.  
Since (\ref{ab2}) holds by assumption, we have $A=B^*\lceil \cD(B)$ and therefore   $\|Bx\|=\|Ax\|$ for $x\in \cD(A)=\cD(B)$. 

Hence, in particular, the operator $A$ is also formally normal and closed. In fact, the above assumption is symmetric in the operators $B$ and $A$.

Let  $x\in \cD(A^*)$ and $ y\in \cD(B^*)$. As noted above (see (\ref{das1}) and (\ref{das2})),  $x$ and $y$ are of the form 
\begin{align}\label{xform1a}
x=x_0+(R^*)^{-1}v_1+u_1,~~~ y=y_0+R^{-1}u_2+v_2,
\end{align}
where $x_0\in \cD(B),\, y_0\in \cD(A),\, v_1,v_2\in \cN(B^*),\, u_1,u_2\in\cN(A^*)$. Then, setting $x'=0, y=0$ and renaming $y'$ by $y$ in formula (\ref{das7}) we obtain
\begin{align}
 \langle A^*x,y\rangle-\langle x, B^*y\rangle
  =\langle v_1,v_2\rangle-\langle u_1, u_2\rangle .\label{das7a}
\end{align}

Now we suppose that  $\cC$ is a closed subspace of $\cK=\cN(A^*)\oplus \cN(B^*)$. We define linear operators $T_\cC$ and $S_{\cC}$ on the Hilbert space  $\cH$ by
\begin{align}\label{defts}
T_\cC =A^*\lceil \cD(T_\cC),\quad S_{\cC}=B^*\lceil \cD(S_{\cC}),
\end{align}
where
\begin{align}\label{tcsc1}
\cD(T_\cC)&=\{ x_0+(R^*)^{-1}v_1+u_1: x_0\in \cD(B),~ (u_1,v_1)\in \cC\, \},\\
\cD(S_{\cC})&= \{y_0+ R^{-1}u_2+v_2: y_0\in \cD(A),~ (u_2,v_2)\in \cC\}.\label{tcsc2}
\end{align}
Further, let $\cC'$ denote the closed linear subspace  of $\cK$ given by 
\begin{align}\label{defcp}
\cC'=\{(u_2,v_2)\in \cK:\langle v_1,v_2\rangle=\langle u_1, u_2\rangle~~ \textit{for~ all}~~ (u_1,v_1)\in \cC \}.
\end{align}
Recall that for any relation $\cC$ on $\cK=\cN(A^*)\oplus\cN(B^*)$ the adjoint relation $\cC^*$ is the relation  on $\cK':=\cN(B^*)\oplus\cN(A^*)$ defined by 
\begin{align}\label{defad}
\cC^*=\{ (v_2,u_2)\in \cK': \langle v_2,v_1\rangle=\langle u_2,u_1\rangle ~~\textit{for}~~ (u_1,v_1)\in \cC\, \}.
\end{align}
Comparing (\ref{defcp}) and (\ref{defad}) we conclude that
\begin{align}\label{cprimstar}
\cC'=\{ (u_2,v_2): (v_2,u_2)\in \cC^* \}.
\end{align} Hence, since $\cC^*$ is a closed  linear relation, $\cC'$ is a closed linear subspace of $\cK$.

\begin{lem}\label{tsc}
$T_\cC$ is a proper extension of $B$, that is, $T_\cC$ is a closed linear operator such that $B\subseteq T_\cC\subseteq A^*$. Each proper extension of $B$ is of this form. Further,
\begin{align}\label{tsadjoint}
(T_\cC)^*=S_{\cC'}\quad {\rm and}~~~(S_{\cC'})^*=T_{\cC}.
\end{align}
\end{lem}
\begin{proof}
By a general result on  boundary triplets  \cite[Lemma 14.13]{sch12}, the mappings $\Gamma_0, \Gamma_1$ of $\cD(\gA^*)$, endowed with the graph norm, into $\cK$ are continuous. Since the operators $A$, $B$ are closed and the subspaces $\cC$, $\cC'$ of $\cK$ are closed, it follows easily from this result that  $T_\cC$ and $S_{\cC'}$ are closed operators. The inclusions $B\subseteq T_\cC\subseteq A^*$ are obvious from the definition of $T_\cC$. Thus, $T_\cC$ is a proper extension of $B$.

Now let $T$ be an arbitrary proper extension of $B$. Then the matrix $\gT$, defined by (\ref{matrixt}) with $S:=T^*$, is a self-adjoint operator on $\cH_2$. Hence, by Lemma \ref{boungc}, $\gT=\gA_\gC$ for some closed relation  $\gC$ on $\cK_2$. Let $\cC$ denote the set of vectors $\Gamma_0(x,y)$ for $(\Gamma_0(x,y),\Gamma_1(x,y))\in \gC$. Then $\cC$ is a closed subspace of $\cK$ and  from the definition of $\gT=\gA_\gC$ it follows that $T=T_\cC$.  

Finally, we prove (\ref{tsadjoint}). From (\ref{das7a}), (\ref{defts}) and (\ref{defcp}) we conclude that
\begin{align}\label{relabstar} 
 \langle T_\cC x,y\rangle=\langle x, S_{\cC'}y\rangle ~~~ \textit{ for} ~~ x\in \cD(T_\cC),~y\in\cD(S_{\cC'}).
 \end{align}
 This equation implies that $T_\cC\subseteq (S_{\cC'})^*$ and $S_{\cC'}\subseteq (T_\cC)^*$. Combining both relations yields
 $ T_\cC \subseteq  (S_{\cC'})^*\subseteq ((T_\cC)^*)^*=T_\cC$, because the operator $T_\cC$ is closed. Thus $T_\cC=(S_{\cC'})^*$. Applying the adjoint yields
 $(T_\cC)^*=((S_{\cC'})^*)^*=S_{\cC'}$.
 \end{proof}
 
 The following theorem characterizes the case when the operator $T_\cC$ is normal. Condition (i) ensures that the  domains $\cD(T_\cC)$ and $\cD((T_\cC)^*)$ coincide, while condition (ii) implies the equality of norms $\|T_\cC z\|$ and $\|(T_\cC)^*z\|$.
 \begin{thm}\label{genform} Suppose  $B$ is a closed formally normal operator and $\cD(A)=\cD(B)$. Let $\cC$ be a closed linear subspace of $\cK$ satisfying the following two conditions:
\begin{itemize}
\item[~ (i)]\, $\{(R^*)^{-1}v_1+u_1:  (u_1,v_1)\in \cC\, \} =\{ R^{-1}u_2+v_2: (u_2,v_2)\in \cC' \, \},$
\item[(ii)]\, If $(u_1,v_1)\in \cC, (u_2,v_2)\in \cC'$ and $(R^*)^{-1}v_1+u_1= R^{-1}u_2+v_2$, then $\|v_1\|=\|u_2
\|$.
\end{itemize}
Then $T_\cC$ is a normal operator on $\cH$ such that $B\subseteq T_\cC\subseteq A^*$ and $A\subseteq (T_\cC)^*\subseteq B^*$. Each normal extension of $B$  on $\cH$ is of this form.
 \end{thm}
 \begin{proof}
 First we prove that the conditions (i) and (ii) imply that $T_\cC$ is a normal operator. Since $(T_\cC)^*=S_{\cC'}$ by Lemma \ref{tsc}, condition (i) ensures that $\cD(T_\cC)=\cD((T_{\cC})^*)$. 
 
 Let $z\in\cD(T_\cC)=\cD((T_\cC)^*)$. Then, by (\ref{tcsc1}) and (\ref{tcsc2}), the vector $z$ is of the form
 \begin{align}\label{formz}
z= x_0+(R^*)^{-1}v_1+u_1=y_0+ R^{-1}u_2+v_2,
\end{align}
 with $x_0\in \cD(B),y_0\in \cD(A)$ and $(u_1,v_1)\in \cC, (u_2,v_2)\in \cC'.$ By condition (i), $R^{-1}u_2+v_2$ is of the form $R^{-1}u_2+v_2=(R^*)^{-1}v_1'+u_1'$ for some vector $(u_1',v_1')\in \cC$. Then, by (\ref{formz}),
 \begin{align*}
 x_0-y_0 + (R^*)^{-1}(v_1-v_1')+(u_1-u_1')=0.
 \end{align*}
 Since  the decomposition $\cD(A^*)=\cD(B)\dot{+} (R^*)^{-1}\cN(B^*) \dot{+}\cN(A^*)$ is a direct sum, we conclude  that $x_0=y_0$.
 
 We have
 $T_\cC z=A^*z=Bx_0+v_1$ and $(T_\cC)^* z=S_{\cC'}z=B^*z= Ay_0+u_2=Ax_0+u_2.$ Using that $B^*v_1=0$ we compute
 \begin{align}\label{norb1}
 \|T_\cC z\|^2 =\langle Bx_0+v_1,Bx_0+v_1\rangle =\|Bx_0\|^2+\|v_1\|^2
 \end{align}
 Similarly, by  $A^*u_2=0$, 
 \begin{align}\label{nora1}
  \|(T_\cC)^* z\|^2 =\langle Ax_0+u_2,Ax_0+u_2\rangle =\|Ax_0\|^2+\|u_2\|^2.
  \end{align}
 By assumption, $B$ is formally normal and $A=B^*\lceil\cD(B)$. Hence $\|Bx_0\|^2=\|Ax_0\|^2$. Since  $\|v_1\|=\|u_2\|$ by condition (ii),  comparing (\ref{norb1}) and (\ref{nora1}) yields $\|T_\cC z\|=  \|(T_\cC)^* z\|$. This proves that $T_\cC$ is normal. By definition, $\cD(B)\subseteq \cD(T_\cC)$ and $B=T_\cC\lceil \cD(B)$. Thus, $T_\cC $ is a normal extension of $B$.
 
 Now we prove the last assertion. Suppose  $N$ is an arbitrary normal extension of $B$ on $\cH$. Clearly, $B\subseteq N$ implies $N^*\subseteq B^*$, so that $N^*\lceil \cD(A)=B^*\lceil \cD(A)=A$ and hence $A\subseteq N^*\subseteq B^*$. Taking the adjoint operation we get $B\subseteq N\subseteq A^*$. 
 
 Since $N$ is a closed operator such that $B\subseteq N\subseteq A^*$, it follows from Lemma \ref{tsc} that $N=T_\cC$ for some closed subspace $\cC$ of $\cK$.
  Then  $\cD(T_\cC)=\cD((T_{\cC})^*)=\cD(S_{\cC'})$ by the normality of $N=T_\cC$ and (\ref{tsadjoint}). From (\ref{xform1a}) it follows the elements of $\cD(T_\cC)=\cD(S_{\cC'})$ are precisely the vectors $z$ of the form  
  \begin{align}\label{formz}
z= x_0+(R^*)^{-1}v_1+u_1=y_0+ R^{-1}u_2+v_2,
\end{align}
 with $x_0,y_0\in \cD(B)=\cD(A)$ and $(u_1,v_1)\in \cC, (u_2,v_2)\in \cC'.$
 
The crucial step is to show that $x_0=y_0$. Recall that
 $T_\cC =A^*z=Bx_0+v_1$ and $(T_\cC)^* z=S_{\cC'}z=B^*z= Ay_0+u_2.$ 
 Therefore, since $T_\cC$ is normal, using equation (\ref{polar}) we derive
 \begin{align*}
\langle & Bx_0+v_1, B(x_0-y_0)\rangle =\langle T_\cC z,T_\cC(x_0-y_0)\rangle \\ & =\langle (T_\cC)^* z,(T_\cC)^*(x_0-y_0)\rangle =\langle Ay_0+u_2, A(x_0-y_0)\rangle.
\end{align*}
By $B^*v_1=0$ and $A^*z_2=0$  this implies 
\begin{align}\label{bain}
\langle Bx_0, B(x_0-y_0)\rangle =\langle Ay_0, A(x_0-y_0)\rangle.
\end{align}
By assumption, the operator $B$ is  formally normal and $A\subseteq B^*$. Hence, combining (\ref{bain}) and (\ref{polar}), we obtain $$\langle Bx_0, B(x_0-y_0)\rangle =\langle By_0, B(x_0-y_0)\rangle,$$ so $\langle B(x_0-y_0),B(x_0-y_0)\rangle =0$ and hence $B(x_0-y_0)=0$. Since $B\subseteq R^*$ and $R^*$ is invertible, we conclude that $x_0=y_0$.

Inserting the equality $x_0=y_0$ into (\ref{formz}) we get $(R^*)^{-1}v_1+u_1= R^{-1}u_2+v_2$. Note that all vectors $(R^*)^{-1}v_1+u_1$ with $(u_1,v_1)\in \cC$ and all vectors $R^{-1}u_2+v_2$ with $(u_2,v_2)\in \cC'$ are in $\cD(T_\cC)=\cD((T_\cC)^*)$. Thus, it follows that condition (i) is fulfilled.

By equations (\ref{norb1}) and (\ref{nora1}) and the normality of the operator $T_\cC$, we have
\begin{align*}
 \|Bx_0\|^2+\|v_1\|^2=\|T_\cC z\|^2 =
  \|(T_\cC)^* z\|^2  =\|Ax_0\|^2+\|u_2\|^2.
  \end{align*}
Recall that\, $\|Bx_0\|=\|Ax_0\|$, because $B$ is formally normal.
Hence $\|v_1\|=\|u_2\|$, which proves that condition (ii) holds.
\end{proof}

Now we consider the special case where 
$\cC$ is the the graph of a {\it densely defined closed linear operator} $C$ of the Hilbert space $\cN(B^*)$ into the Hilbert space $\cN(A^*)$:
\begin{align}\label{graphc}
\cC=\{(Cv_1,v_1): v_1\in \cD(C)\, \}.
\end{align}

\begin{lem}\label{lemcp}
$\cC'=\{ (u_2,C^*u_2): u_2\in \cD(C^*)\}$.
\end{lem}
\begin{proof}
Let $(u_2,v_2)\in \cK$. By the definitions (\ref{graphc}) and (\ref{defcp}) of $\cC$ and  $\cC'$, we have $(u_2,v_2)\in \cC'$ if and only if $\langle v_1,v_2\rangle= \langle Cv_1,u_2\rangle$ for all $v_1\in \cD(C)$. The latter holds if and only if $u_2\in \cD(C^*)$ and $v_2=C^*u_2$, which proves the assertion.
\end{proof}

The following is a reformulation of Theorem \ref{genform} for subspaces (\ref{graphc}). 
\begin{thm}\label{graphform} Suppose  $B$ is a closed formally normal operator and $\cD(A)=\cD(B)$. 
Let $\cC$ be a closed subspace of $\cK$ of the form (\ref{graphc}). 

Then the operator $T_\cC$ is  normal  if and only if there exists an isometric linear operator $U$ of $\cD(C^*)$ onto $\cD(C)$ such that 
\begin{align}\label{newi}
R^{-1}u_2+ C^*u_2=(R^*)^{-1}Uu_2+CUu_2 \quad {\rm for}\quad u_2\in \cD(C^*).
\end{align}
 \end{thm}
 \begin{proof}
 First we suppose that $T_\cC$ is normal. Let $u_2\in \cD(C^*)$. Then $(u_2,C^*u_2)\in \cC'$, so by condition (i) there exists a vector $v_1\in \cD(C$ such that 
 $ R^{-1}u_2+v_2=(R^*)^{-1}v_1+u_1$. Since  (\ref{das2}) is a direct sum, $v_1$ is uniquely determined by $u_2$. Clearly, the map $u_2\mapsto v_1$ is linear. Since $\|u_2\|=\|v_1\|$ by condition (ii) and equation (\ref{graphc}),  there is an isometric linear map $U:\cD(C^*)\mapsto \cD(C)$ given by $Uu_2=v_1$. Inserting $Uu_2=v_1, v_2=C^*u_2, u_1=Cv_1=CUu_2$ into condition (i), we obtain (\ref{newi}).
 
Now we prove the converse implication. Let $(u_2, v_2)\in \cC'$. Then $v_2=C^*u_2$ by Lemma \ref{lemcp} and $v_1:=Uu_2\in \cD(C)$, so $(Cv_1,v_1)=(CUu_2,Uv_1)\in \cC$. Then (\ref{newi}) gives $R^{-1}u_2+v_2=(R^*)^{-1}v_1+u_1$. 

Now let $(u_1, v_1)\in \cC$. Since $U$ is surjective, there is $u_2\in \cD(C^*)$ such that $Uu_2=v_1$. Then $(u_2,C^*u_2)\in \cC'$, $u_1=Cv_1$ by (\ref{graphc}) and (\ref{newi}) yields $R^{-1}u_2+v_2=(R^*)^{-1}v_1+u_1$. This proves that condition (i) is satisfied. Since $U$ is isometric, condition (ii) holds as well.
\end{proof}\section{Normal operators}\label{formn2}
In this section, we assume in addition that the {\bf operator $R$ is normal.} 
Recall that we assumed throughout that $R$ has a bounded inverse $R^{-1}\in {\bf B}(\cH)$.

Note that $R$ is normal if and only if $R^*$ is  normal, or equivalently, $R^{-1}$ (resp. $(R^*)^{-1}=(R^{-1})^*$) is normal. In particular, the  assumption is  symmetric in the operators $A$ and $B$.

Next we consider the {\it polar decomposition} of the operator $R$:
 \begin{align}\label{polar}
R=U|R|.
\end{align}
Here $|R|
:=[R^*R)^{1/2}$ is the {\it modulus} of $R$ and  the {\it phase operator} $U$ of $R$ is a partial isometry with initial space $\cN(T)^\bot$ and final space $\cN(T^*)$. Since $R$ is normal with  bounded inverse, $U$ is a unitary operator which commutes with $R|$. We have 
\begin{align}\label{polrela}
|R^*|=|R|,~R^*= U^*|R|, ~ R=U|R|=|R|U, ~R^{-1}=U^*|R|^{-1}, ~(R^*)^{-1}=U|R|^{-1}.
\end{align}

 Since $A\subseteq R$, we have $\cD(A)\subset \cD(R)=\cD(|R|)$. Let $$T:=|R|\, \lceil \cD(A)$$ denote the restriction of $|R|$\, to the domain $\cD(A)$. Then $T$ is a densely defined positive symmetric operator such that $T\geq (\|R^{-1}\|)^{-1}\cdot I$.
  \begin{lem}\label{uab}
 $\cN(A^*)=U\cN(T^*)$ and $\cN(B^*)=U^*\cN(T^*)$.
 \end{lem}\begin{proof}
Let $x\in \cN(T^*)$. Using (\ref{das0}) we derive
\begin{align*}
\langle Ux,Ay\rangle & =\langle x,U^*Ay\rangle=\langle x,U^*Ry\rangle=\langle x,U^*U|R|y\rangle\\ &=\langle x,|R|y\rangle=\langle x,Ty\rangle=\langle T^*x,y\rangle  =0
\end{align*}
for $y\in \cD(A)=\cD(B)$. From this equation it follows that $Ux\in \cD(A^*)$ and $A^*Ux=0$, that is,  $Ux\in \cN(A^*)$. This shows that $U\cN(T^*)\subseteq \cN(A^*)$.

Conversely, suppose $z\in \cN(A^*)$. For $y\in \cD(A)=\cD(B)$, we obtain
\begin{align*}
\langle U^*z, Ty\rangle=\langle U^*z, |R|y\rangle=\langle z,U|R|y\rangle=\langle z,Ry\rangle=\langle z,Ay\rangle=\langle A^*z,y\rangle=0.
\end{align*}
Therefore, $U^*z\in \cD(T^*)$ and $T^*U^*z=0$. That is, we have $U^*z\in \cN(T^*)$ and hence $z\in U\cN(T^*)$, so that $\cN(A^*)\subseteq U\cN(T^*)$. 

Putting the preceding together, we have proved that $ U\cN(T^*)=\cN(A^*)$.

The proof of the second equality $\cN(B^*)=U^*\cN(T^*)$ is similar.
\end{proof}

Next we introduce another unitary operator $W$. Since $R$ is normal,  the equation 
\begin{align}\label{defW}
WR^*x=Rx,\quad x\in \cD(R)=\cD(R^*),
\end{align} 
 defines an isometric linear operator on $\cH$ with dense domain and dense range. Hence it extends to unitary operator, denoted again by $W$, on $\cH$. Then
 \begin{align}\label{wequ}
  WR^*=R,~ W^*R= R^*,~~ {\rm and}~~  (R^*)^{-1}=R^{-1}W, ~W=R(R^*)^{-1}.
 \end{align}
 Applying the adjoint to   $(R^*)^{-1}=R^{-1}W$, we get $R^{-1}=W^*(R^*)^{-1}=W^*R^{-1}W$, so  $WR^{-1}=R^{-1}W$ and $W^*R^{-1}=R^{-1}W$. This implies that the unitary $W$ commutes with $R^{-1}$ and $(R^*)^{-1}$. 
 
 The unitary $W$ is  the square of the phase operator $U$. Indeed, since $WU^*|R|x=U|R|x$ by (\ref{defW}) and the range of $|R|$ is dense, we get $WU^*=U$, so that $$W=U^2. $$ 
 Combined with Lemma \ref{uab} we conclude that
 \begin{align}\label{wnab}
 W\cN(B^*)=\cN(A^*)\quad {\rm and}\quad W^*\cN(A^*)=\cN(B^*).
\end{align}

From (\ref{wequ}) and  (\ref{wnab}) we obtain $
(R^*)^{-1}\cN(B^*)=R^{-1}W\cN(B^*)=R^{-1}\cN(A^*)$. Inserting this into (\ref{das1}) we get
\begin{align}
\cD(A^*)&=\cD(B)\dot{+} R^{-1}\cN(A^*) \dot{+}\cN(A^*),\label{da1}\\ \cD(B^*)&=\cD(A)\dot{+} R^{-1}\cN(A^*) \dot{+}\cN(B^*).\label{db1}
\end{align} 
Further, by (\ref{polrela}), 
\begin{align*}
(R^*)^{-1}\cN(B^*)&=U |R|^{-1}U^*\cN(T^*)=\cN(T^*), \\ R^{-1}\cN(A^*)&=U^* |R|^{-1}U\cN(T^*)=\cN(T^*)
\end{align*}
and therefore by Lemma \ref{uab}, ,
\begin{align}
\cD(A^*)&=\cD(B)\dot{+} |R|^{-1}\cN(T^*) \dot{+}U\cN(T^*),\label{da2}\\ \cD(B^*)&=\cD(A)\dot{+} |R|^{-1}\cN(T^*) \dot{+}U^*\cN(T^*).\label{db2}
\end{align} 
The formulas (\ref{da1}), (\ref{db1}), (\ref{da2}), (\ref{db2}) are useful descriptions of the domains $\cD(A^*)$ and $\cD(B^*).$

Using  the unitaries $W$ and $U$ we  can reformulate and slightly simplify
 Theorems \ref{genform} and \ref{graphform} under the assumption that $R$ is normal. We do not carry our these  restatements and mention only the corresponding changes 
 in the case of $W$. Then condition (i) in Theorem \ref{genform} should be replaced by
\begin{align*}
 \{R^{-1}Wv_1+u_1:  (u_1,v_1)\in \cC\, \} =\{ R^{-1}u_2+v_2: (u_2,v_2)\in \cC' \, \},
 \end{align*}
 and in Theorem \ref{graphform}  equation (\ref{newi}) becomes
 \begin{align*}
R^{-1}u_2+ C^*u_2=R^{-1}WUu_2+CUu_2 \quad {\rm for}\quad u_2\in \cD(C^*).
\end{align*}

\begin{exm}\label{extri}
In this example we consider the special case 
\begin{align*}
\cC=\{(0,v_1): v_1\in \cN(B^*)\, \}.
\end{align*} 
Then $\cC'=\{(u_2,0): u_2\in\cN(A^*)\, \}$. Since $R^{-1}W=(R^*)^{-1}$ and $W\cN(B^*)=\cN(A^*)$, condition (i) of Theorem \ref{genform} is fullfilled. Condition (ii) holds  trivially, so the operator $T_\cC$ is normal. From (\ref{tcsc1}) and (\ref{tcsc2}) we conclude easily that $T_\cC=R^*$ and $S_{\cC'}=R$.
\end{exm}
To construct examples we now reverse our considerations and  begin with  a {\bf bounded normal operator} $Z$ on $\cH$ with trivial kernel. 

Then $\cN(Z^*)=\cN(Z)=\{0\}$ and $R:=Z^{-1}$ is a normal operator with adjoint $R^*=(Z^*)^{-1}$. Further,  assume  that $R$ is  {\bf unbounded}. 
Then $\cD(R)\equiv\cR(Z)\neq \cH$.

 From now on  suppose that $\cU\neq \{0\}$ is a closed linear subspace of $\cH$ such that 
\begin{align}\label{intersec}
 \cU\cap \cD(R)=\{0\}.
\end{align}
 Since $\cD(R)\neq \cH$, such spaces exists; one can even show that there are infinite-dimensional closed subspaces $\cU$ satisfying (\ref{intersec}).
  
 We denote by $P$  the orthogonal projection of $\cH$ on $\cU$ and by $W$ the unitary operator defined by (\ref{defW}). Then  equation
  (\ref{wequ}) holds. In particular, $Z=Z^*W^*$. 
 Further, we define
\begin{align}
&A:=R\lceil \cD(A)~~{\rm and}~~ B:=R^*\lceil \cD(B),\\
~~{\rm where}~~&\cD(A)=\cD(B):=Z(I-P)\cH=Z^*W^*(I-P)\cH.
\end{align}
\begin{prop}
$A$ and $B$ are densely defined closed formally normal operators and $0$ is a regular point for both operators. They form an adjoint pair and we have $\cD(A^*)=P\cH=\cU$ and 
$\cD(B^*)=W^*P\cH=W^*\cU$.
\end{prop}
\begin{proof}
First we show that $\cD(A)$ is dense. Let $x\in \cH$. Assume that $x\perp \cD(A)$. Then we have 
\begin{align*}
0=\langle x,Z(I-P)y\rangle=\langle Z^*x,(I-P)y\rangle 
\end{align*} for all $y\in \cH$, so $Z^*x\in P\cH=\cU$. Since $Z^*x\in \cD(R^*)=\cD(R)$, we obtain $Z^*x=0$ by \ref{intersec}). Hence $x=0$, which proves that $\cD(A)$ is dense.

The operators $A$ and $B$ are formally normal, because they are restrictions of the normal operators $R$ and $R^*$, respectively. Since $R$ and $R^*$ have bounded inverses, $0$ is a regular point for $A$ and $B$. Clearly, $A$ and $B$ form an adjoint pair.

We prove that $\cN(A^*)=P\cH$ and $\cN(B^*)=W^*P\cH$. Let $y\in \cH$. Since 
\begin{align*}\langle A Z(I-P)x,y\rangle =\langle (I-P)x,y\rangle
\end{align*} for all $x\in \cH$, it follows that $y\in \cN(A^*)$ if and only if $y\in P\cH$. Similarly, 
\begin{align*}
\langle BZ^*W^*(I-P)x,y\rangle= \langle(I-P)x,Wy\rangle, \quad x\in \cH,
\end{align*} 
implies that $y\in \cN(B^*)$ if and only if $Wy\in P\cH$, that is, $y\in W^*P\cH$.
\end{proof}
\section{The one-dimensional case}\label{formn3}
In this section we remain the setup and the assumptions of the preceding section. We shall treat the simplest case when $\cU=P\cH=\cN(A^*)$ has dimension one.
Throughout this section, we suppose that $\cU=\dC\cdot \xi$, where $\xi$ is a fixed vector of $\cH$ such that $\xi\notin \cD(R)$. 

Consider a linear subspace $\cC$ of\, $\cK=\dC\cdot\xi \oplus \dC\cdot W^*\xi$. 
It is obvious that the operator $T_\cC$ is not normal if $\dim \cC=0$ or $\dim\cC=2$. 
If $\cC=\dC\cdot W^*\xi$,  we know from Example (\ref{extri}) that $T_\cC=(R^*)^{-1}$.
Thus it remains to study that case 
\begin{align}\label{assuc}
\cC=\dC\cdot (\xi,\alpha W^*\xi)\quad {\rm for~ some}~~ \alpha\neq 0.
\end{align}
Then, by (\ref{defcp}), a vector $ (\beta_1\xi,\beta_2W^*\xi)\in \cK,$ with  $\beta_1,\beta_2\in \dC$, belongs to $\cC'$ if and only if 
$\langle \alpha W^*\xi, \beta_2 W^*\xi\rangle= \langle \xi,\beta_1 \xi\rangle$,
 or equivalently, $\ov{\alpha}\beta_2=\beta_1$ . Therefore,
\begin{align}\label{cprime}
\cC'=\dC\cdot(\ov{\alpha}\, \xi, W^*\xi).
\end{align}
Note that $R^{-1}=Z$ and $(R^*)^{-1}=Z^*$. Hence, from (\ref{assuc}) and (\ref{cprime}) it follows that condition (i) of Theorem \ref{genform} holds if and only if that there exists a number $\gamma\in \dC, \gamma\neq 0,$ such that 
 \begin{align}\label{cprime0}
 Z^*( \alpha W^*\xi) +\xi=Z(\gamma\, \ov{\alpha}\, \xi)+ \gamma W^*\xi.
 \end{align} 
 Clearly,  condition  (ii) is equivalent to $\|\alpha W^*\xi\| =\|\gamma\, \ov{\alpha}\, \xi\|$, that is, $|\gamma|=1$.
Recall that $Z^*W^*=Z$. Therefore, by the preceding,   Theorem \ref{genform}  yields the following: 

{\it $T_\cC$ is a normal operator if and only if there is a  number $\gamma\in \dC$  such that}
 \begin{align}\label{noreqt}
 (\alpha- \ov{\alpha}\, \gamma) Z\xi &=(\gamma W^*-I)\xi,\\
 |\gamma|&=1.\label{noregt1}
 \end{align}
 
 Before we continue we illustrate this statement in a very special case.
 \begin{exm}
 Suppose  $S$ is a densely defined symmetric operator with equal non-zero
 deficiency indices. Then $S$ has a self-adjoint extension $X$ on  $\cH$. Clearly, $A:=S+\ii I$ and $B:=S-\ii I$ are formally normal operators with domain $\cD(S)$ and $R:=X+\ii I$ is a normal extension of $A$ with bounded inverse $Z$. We choose a vector $\xi\in \cH$ such that $\xi \notin\cD(X)$ and define $\cC$ and $\cC'$  by (\ref{assuc}) and (\ref{cprime}), respectively. Then we are in the setup described above.
  
Let us consider equation (\ref{noreqt}). Since $W^*-I=(X-\ii I)(X+\ii I)^{-1}$, we have 
$$(\gamma W^*-I)\xi=-2\ii \gamma (X+\ii I)^{-1}\xi +(\gamma-1)\xi=-2\ii \gamma Z\xi +(\gamma-1)\xi.
$$
Therefore, since $Z\xi \in \cD(X)$ and $\xi\notin \cD(X)$, (\ref{noreqt}) is fulfilled if and only if $\gamma=1$ and $\alpha- \ov{\alpha}=-2\ii$, or equivalently, $\gamma =1$ and $\alpha=a-\ii$ with $a$ real. Therefore, by the preceding,  the operator $T_\cC$ is  normal if and only if 
\begin{align}\label{norextsp}
\cC=\dC\cdot (\xi, (a-\ii)W^*\xi)\quad \textit{for some}~~~a\in \dR.
\end{align} That is, the normal extensions of $B$ are parametrized by the real number $a$. It can be shown that for $\cC$ as in (\ref{norextsp}) the corresponding operator $T_\cC+\ii I$ is self-adjoint and hence a self-adjoint extension of the symmetric operator $S=B+\ii I$.
 \end{exm}

 Now we return to the general case. 
 The normal operator $R=Z^{-1}$ can be written as $R=X+\ii Y$, where $X$ and $Y$ are strongly commuting self-adjoint operators. Since $R$ is unbounded, at least one of the operators $X$ and $Y$ has to be unbounded. Our aim is to reformulate conditions (\ref{noreqt}) and (\ref{noregt1}) in terms of  $X$ and $Y$. For this we need some  notation.

Suppose $|\gamma|=1$ and $\gamma\neq 1$. Then we define numbers $t_\gamma$ and $s_\gamma$ by
\begin{align}
 t_\gamma &:=\ii (\gamma+1)(\gamma- 1)^{-1},\label{deftg}\\
  s_{\gamma,\alpha} &:=(\alpha-\ov{\alpha}\, \gamma )(\gamma- 1)^{-1}.\label{defsg}
  \end{align}
 Both numbers $t_\gamma$ and $s_{\gamma,\alpha}$ are real, because $|\gamma|=1$. In the case $\gamma=1$
 we set
 \begin{align}\label{infty}
 t_1:=\infty, ~~~ s_{1,\alpha}:={\rm Im}\, \alpha.
 \end{align} Formulas (\ref{deftg}) and (\ref{defsg}) express  the real numbers $t_\gamma$ and $s_{\gamma,\alpha}$ in terms of $\gamma$ and $\alpha$. 
 
 Now we want to reverse these transformations.  
Let $\alpha =a+\ii b$ with $a,b$ real. Then we obtain
\begin{align}
\gamma&=(t_{\gamma}+\ii)(t_\gamma-\ii)^{-1},\label{defg}\\
\alpha&= a+\ii b= bt_\gamma -s_{\gamma,\alpha} +\ii b.\label{defs}
\end{align}
Thus,  given  $t,s\in \dR$, $\gamma$ is uniquely determined by (\ref{defg}) and there is a one-parameter family of numbers $\alpha$ in (\ref{defs}), with $b\in \dR$ as real parameter, satisfying the equations (\ref{deftg}) and (\ref{defsg}). Likewise, if $t=\infty, s\in \dR$, we have $\gamma=1$ and $\alpha =a+\ii s$ is a one-parameter family with real parameter $a$ such that (\ref{infty}) holds.

 Further, we define operators
 \begin{align}
 R_t:=X-t Y, ~t\in \dR, ~~~R_\infty:=-Y.
 \end{align}
 Note that $R_\infty$ and the closure of $R_t$, $t\in \dR$, are self-adjoint operators.
 \begin{thm}\label{onedimcase} Suppose  $\cC$ is a  subspace of $\cK$ which is of the form (\ref{assuc}). 
 Then the operator 
 $T_\cC$ is a normal operator if and only if there exists a $t\in \dR\cup\{\infty\}$ such that $\xi$ is  eigenvector of the operator $R_t$. 
 In this case, if $s$ denotes the eigenvalue of eigenvector $\xi$ of $R_t$,  the pair $(t,s)$ is uniquely determined  by $\cC$. 
More precisely, if conditions  (\ref{noreqt}) and (\ref{noregt1}) are satisfied, then $t=t_\gamma$ and $s=s_{\gamma,\alpha}$. 
 \end{thm}
 \begin{proof}
 First we rewrite the right-hand side of equation (\ref{noreqt}).
From the definition of $W$ we get $W^*=R^*Z$, so $W^*=(X-\ii Y)(X+\ii Y)^{-1}$ and therefore
 \begin{align}
 \gamma W^*- I&=\gamma (X-\ii Y)(X+\ii Y)^{-1}- (X+\ii Y)(X+\ii Y)^{-1}\nonumber \\ &
 =[(
 \gamma- 1)X-\ii (\gamma+1)Y](X+\ii Y)^{-1}.\label{wequc} 
 \end{align}
 
 Suppose $\gamma\neq 1$. Then, combining (\ref{wequc}) and (\ref{deftg}) we  obtain
 \begin{align*}
 \gamma W^*- I= (\gamma -1)[X- t_\gamma Y]((X+\ii\, Y)^{-1}=(\gamma -1)R_{t_\gamma}(X+\ii\, Y)^{-1}, 
 \end{align*}
  Hence, since $(\alpha -\ov{\alpha}\, \gamma) Z=s_{\gamma,\alpha} (\gamma-1)(X+\ii Y)^{-1}$, condition (\ref{noreqt}) is equivalent to
 \begin{align}\label{inves}
 (s_{\gamma,\alpha} -R_{t_\gamma})(X+\ii Y)^{-1}\xi=0. \end{align}
  
  Next we show that (\ref{inves}) implies that
  \begin{align}\label{eigenn}
  R_{t_{\gamma}}\xi =s_{\gamma,\alpha}\xi.\end{align}
  Let $E$ denote the spectral measure of the unbounded normal operator $R=X+\ii Y$.
 Set $\mu_\xi$ denote the measure $\langle E(\cdot)\xi,\xi\rangle$ on $\dC$. From the functional calculus for self-adjoint operators we obtain
\begin{align}
\int_\dC \big| (s_{\gamma,\alpha}- (x-t_\gamma y))(x+\ii y)^{-1}\big|^2~ d\mu_\xi(x+\ii y ) =0
 \end{align}
This implies that the function $ (s_{\gamma,\alpha}- (x-t_\gamma y))(x+\ii y)^{-1}$ is zero $\mu_\xi$-everywhere. Hence $ s_{\gamma,\alpha}- (x-t_\gamma y)=0$ $\mu_\xi$-everywhere on $\dC$. Therefore, 
\begin{align}\label{eigen3}
\int_\dC \big| s_{\gamma,\alpha}- (x-t_\gamma y)\big|^2~ d\mu_\xi(x+\ii y ) =0
 \end{align}
Again by the spectral calculus, the integral on the left is 
\begin{align*}
\|(s_{\gamma,\alpha}-(X-t_\gamma Y))\xi\|^2=\|(s_{\gamma,\alpha}-R_{t_{\gamma}})\xi\|^2.
\end{align*}
 Thus, by (\ref{eigen3}),  $(s_{\gamma,\alpha}-R_{t_\gamma})\xi=0$, which proves (\ref{eigenn}).

 Now we consider the case $\gamma=1$. Then, since $W^*-I=-2\ii\, Y(X+\ii\, Y)^{-1}$ and $(\alpha-\ov{\alpha})Z=
 2\ii\, ({\rm Im}\, \alpha) (X+\ii Y)^{-1}$, equation (\ref{noreqt}) reads as 
 \begin{align}\label{eigen4}
 (({\rm Im} \, \alpha)+ Y)(X+\ii Y)^{-1}\xi=0.
 \end{align}
 Repeating the reasoning of the preceding paragraph, 
  (\ref{eigen4}) yields $(({\rm Im} \, \alpha)+ Y)\xi=0$. This means that $(s_{1,\alpha}-R_{t_1})\xi=0$, which proves (\ref{eigenn}) in the case $\gamma=1$.  
  
  Summarizing, we have shown that conditions  (\ref{noreqt}) and (\ref{noregt1})  imply $R_{t_{\gamma}}\xi =s_{\gamma,\alpha}\xi$ for all nonzero $\alpha$ and $\gamma$ of modulus one. Conversely, reversing the above reasoning it follows from the equation $R_{t_{\gamma}}\xi =s_{\gamma,\alpha}\xi$ that  (\ref{noreqt}) and (\ref{noregt1})  hold.  
  
  It remains to show that the pair $(t,s)$, where $t\in \dR\cup\{\infty\}$, $s\in \dR$, is uniquely determined by the subspace $\cC$. Assume that  $R_t\xi =s\xi$ and $R_{t'}\xi =s'\xi$. 
  
  First let $t=\infty$. Assume to the contrary that $t'\in \dR$.  Then it follows that $\xi \in \cD(Y)$ and $\xi\in \cD(X+tY)$, so $\xi \in \cD(R)$,  which contradicts the assumption. Thus $t'=\infty.$ Then $-Y\xi=s\xi$ and $-Y\xi=s'\xi$ imply $s=s'$.
  
  Now suppose $t\in \dR$. Then, as shown in the preceding paragraph,  $t'\in \dR$. Again we assume to the contrary that $t\neq t'$. Then $(R_t-R_{t'})\xi=(t'-t)Y\xi=(s-s')\xi$, so $Y\xi=(t'-t)^{-1}(s-s')\xi$ and $(X+\ii Y)\xi=(X-tY)\xi +(\ii +t)Y\xi=c\xi$ with $c=s+(\ii+t)(t'-t)^{-1}(s-s')$. Arguing as in the preceding paragraph leads to a contradiction. Thus $t=t'$ and hence also $s=s'$. 
  Thus we have shown in all cases that $(t,s)=(t',s')$.
 \end{proof} 
 We illustrate  the preceding theorem by an example.
 \begin{exm}\label{fnnor}
 Let $\mu$ be a Radon measure on $\dC$  such that $\mu(\{z\in \dC: |z|\leq \varepsilon\})=0$ for some $\varepsilon >0$ and let $R$ denote the  multiplication operator by the complex variable $z$ on $\cH=L^2(\dC;\mu)$. Clearly, $R$ and $R^*$ are normal operators with bounded inverses. 
 We suppose that $R$ is unbounded and choose a function $\xi(z)\in L^2(\dC;\mu)$ such that $z\xi(z)\notin L^2(\dC;\mu)$. Then we are in the setup discussed above and $R^*$ is a normal extension of the operator $B$.
 By appropriate choices of $\mu$ and $\xi$ we can construct interesting cases.

 First fix $s,t\in \dR$ and suppose that $\mu$ is supported on the line $x-ty=s$, where $z=x+\ii y, x, y\in \dR$. Then $(X-tY)\xi =s\xi$.
 Define $\gamma$ by  (\ref{defg}), $\alpha$ by (\ref{defs}) with  $b\in\dR$, and the vector space $\cC$ by (\ref{assuc}). Then, by Theorem \ref{onedimcase}, $T_\cC$ is normal and the operator $B$ 
  has, in addition to $R^*$, precisely the one-parameter family of operators $T_\cC$ (with real parameter $b$) as normal extensions on the Hilbert space $\cH$.
  
   A similar result is true for $t=\infty$, $s\in \dR$.
  
  Next we choose $\mu$ and $ \xi$  such $\xi$ is not an eigenvector of some operator $R_t$ with $t\in \dR\cup \{\infty\}$. (For instance,  let $\mu$ be the Lebesgue measure outside of some ball around the origin and set $\xi:=|z|^{-3}$.) Then the operator 
  $B$ has no other normal extension on $\cH$ than the operator $R^*$.
 \end{exm}

\bibliographystyle{amsalpha}

\end{document}